\documentclass[12pt]{article}

\usepackage[utf8]{inputenc}
\usepackage[T1]{fontenc}
\usepackage[a4paper,left=2cm,right=2cm,top=2.5cm,bottom=2cm]{geometry}
\usepackage{amsfonts} 
\usepackage{amsmath} 
\usepackage{amssymb} 
\usepackage{amsthm} 
\usepackage[numbers]{natbib}
\usepackage{babel}
\usepackage{hyperref} 
\usepackage{cleveref} 
\usepackage{dsfont} 
\usepackage{enumitem}
\usepackage{graphicx}
\usepackage{pifont}
\usepackage{tabto}
\usepackage{wrapfig}

\usepackage{authblk}

\theoremstyle{plain}
\newtheorem{thm}{Theorem}[section]
\newtheorem{lem}[thm]{Lemma}

\theoremstyle{definition}
\newtheorem{dfnt}{Definition}[section]

\theoremstyle{remark}

\crefname{thm}{theorem}{theorems}
\crefname{lem}{lemma}{lemma}
\crefname{prop}{proposition}{propositions}
\crefname{cor}{corollaire}{corollaires}
\crefname{dfnt}{definition}{definitions}
\crefname{exmp}{exemple}{exemples}
\crefname{xca}{exercice}{exercices}
\crefname{rmq}{remarque}{remarques}
\crefname{note}{note}{notes}
\crefname{case}{case}{cases}
\crefname{hyp}{hypothese}{hypothèse}

\newcommand{\Tgn}{\mathcal{T}_{g}}
\newcommand{\PTgn}{P^1 \mathcal{T}_{g}}
\newcommand{\Mgn}{\mathcal{M}_{g}}
\newcommand{\PMgn}{P^1 \mathcal{M}_{g}}
\newcommand{\hfem}{\hat{f}_{\epsilon_0,\mu}}
\newcommand{\fem}{f_{\epsilon_0,\mu}}
\newcommand{\ML}{\mathcal{ML}}

\begin{document}

\title{Upper bound on the rate of mixing for the Earthquake flow on moduli spaces}
\author[1]{Bonnafoux, Etienne}
\affil[1]{Centre de mathématiques Laurent-Schwartz, Ecole polytechnique}
\maketitle

\begin{abstract}
  We prove that the earthquake flow is at most polynomially mixing with a degree bounded by a constant depending only on the topology of the surface. In particular it is not exponentially mixing.
\end{abstract}

\section{Introduction}
\subsection*{Motivation}

The earthquake flow is an object introduced by Thurston which was used by Kerckhoff \cite{NielsenRealizationPro} to solve the Nielsen realisation problem of describing finte subgroup of mapping class group by hyperbolic isometries.

From the point of view of ergodic theory, the earthquake flow is a mysterious object despite tough its connections with other areas of mathematics : see \cite{wright2021mirzakhanis} for a survey. Nonetheless, Mirzakhani \cite{Mirzakhani2010ErgodicTO} sucessfully compared it to the so-called Teichmüller horocyclic flow :

\begin{thm}[Mirzakhani]
There is measurable conjugacy between the earthquake flow on the unit lamination bundle of the moduli space $E^t : P^1 \mathcal{M}_{g} \to P^1 \mathcal{M}_{g}$ with a canonical measure $\nu$,  and the Teichmüller horocyclic flow $U^t : P^1 \mathcal{Q}_{g} \to P^1 \mathcal{Q}_{g} $ on the unit area quadratic differentials bundle over the moduli space with the Masur-Veech measure $\mu_{MV}$ .
\end{thm}

An immediate corollary of this theorem is that the earthquake flow is mixing with respect to the canonical measure $\nu$. However, this conjugacy is not smooth enough (see Mirzakhani \cite{Mirzakhani2010ErgodicTO}, Calderon and Farre \cite{calderon2021shearshape}, Arana-Herrera and Wright \cite{AranaEtWright})
 to allow the computation of the rate of mixing of the earthquake flow by comparison with the one of the Teichmüller horocyclic flow (which is known to be polynomial after works of Ratner \cite{ratner_1987}, Avila, Gouezel and Yoccoz \cite{avila2006exponential} and Avila and Resende \cite{avila2012exponential}).

\subsection*{Results}

In this paper we exclude very fast rates of mixing of the earthquake flow. In fact we prove that the earthquake flow is at most polynomially mixing.

\begin{thm}
 \label{Thprincipal}
 Suppose that there are constants $d$  and $C$ so that
 \begin{equation}
   \label{eq2}
   \left\vert \int f g\circ E^t d \nu - \int f d \nu \int g d \nu \right\vert < C \frac{1}{t^{d}}  \|f \|_{Lip}  \|g \|_{Lip}
 \end{equation}
 for all Lipschitz functions $f,g : P^1 \mathcal{M}_{g} \to \mathbb{R}$ and all times $t \geq 1$ (where $\| \cdot \|_{Lip}$ will be defined in Subsection \ref{Norm_Lip} below). Then $d \leq 6g-5$.
\end{thm}

We prove our results using a strategy similar to that used by K. Burns et al. \cite{RateWPflow} to show that the rate of mixing of the Weil-Petersson geodesic flow is not fast in general. More concretely, assuming that the earthquake flow is polynomially mixing, we exhibit a family of test functions whose supports have relatively large volumes while taking a long time to reach the thick part of the moduli space.

\subsection*{Structure of the paper}

In Section \ref{NaB}, we brifely recall the definitions and the basics properties of the earthquake flow and its phase space. After that, we prove Theorem \ref{Thprincipal} in Section \ref{RoM} by constructing adequate test functions thanks to the strong non-divergence property of the earthquake flow established by Minsky and Weiss in \cite{articleMetW}.

\section{Preliminaries}
\label{NaB}

\subsection{Notations}

Let $S$ be a surface of genus $g$.
The Teichmüller space $\Tgn$ is the space of hyperbolic metrics marked by $S$. The mapping class group $Mod(S)$ is the quotient of the group of orientation preserving diffeomorphisms by the identity component subgroup. It acts on $\Tgn$ and the quotient $\Tgn / Mod(S) = \Mgn$ is called the moduli space of curves of genus $g$.

 We denote by $\ML$ the space of measured laminations and by $\ML(\mathbb{Z})$ the space of multicurves on the surface $S$.
  Let $\Gamma $ be the set of free homotopy equivalence classes of simple closed curves. For a lamination $\lambda$ and a hyperbolic metric $X \in \Tgn$, let $l_\lambda(X)$ be the hyperbolic length.
  We denote by $\PTgn := \{ (X,\lambda) \in \Tgn \times \ML, l_{\lambda}(X)=1 \} $ the unit sub-bundle of measured lamination over the Teichmüller space
 and $\PMgn$ the quotient $\PTgn / Mod(S)$.

Given a pants decomposition $\gamma_1,\dots,\gamma_{3g-3}$, we have a natural symplectic $2$-form on $\Tgn$ : $$
d \omega_{WP} = \sum_{i=1}^{3g-3} dl_{\gamma_i} \wedge d \tau_{\gamma_i},
$$
 where $l_{\cdot}$ the length function and $\tau_{\cdot}$ the twist. As it was discovered by Wolpert \cite{Wolpert1,Wolpert2}, this symplectic form does not depend on the choice of pants decomposition, and it induces a volume form $\mu_{WP}$ called the Weil-Petersson measure.

Moreover, $\ML$ possesses a family of measures depending on $X \in \Tgn$ :
$$
\mu_{Th}'(X)(A) = \underset{L \to \infty}{\lim} \frac{ \# \{ \delta \in \Gamma \cap A, l_{\delta}(X) \leq L\}}{L^{6g-6}}.
$$
These measures are called Thurston measures and they can be projected to $P^1 \ML = \ML / \mathbb{R}^{>0}$ by setting:
$$
\mu_{Th}(X)(A)=\mu_{Th}'(X)(\{ \lambda \in \ML, l_{\lambda}(X) \leq 1,  [\lambda] \in A \}).
$$
The product measure $\mu_{WP}\times \mu_{Th}$ on $\PTgn$ is $Mod(S)$-invariant and, hence, it induces a measure on $\PMgn$ called $\nu$.
\subsection{The Earthquake flow}

The earthquake flow $\hat{E}^t(\cdot,\cdot): \PTgn \rightarrow \PTgn$ is defined in the following way:
\begin{itemize}
  \item For $(X,\gamma) \in \Tgn \times \Gamma$, $\hat{E}^t(X,\gamma)$ is a twist around $\gamma$ of length $t$.
  \item For $X \in \Tgn$, $\sum_{i=1}^n \gamma_i \in \ML(\mathbb{Z})$ and $(a_i) \in \mathbb{R}^n$,  then
  $$\hat{E}^t \left( X,\sum a_i \gamma_i \right) = \hat{E}^{a_1t}(\cdot,\gamma_1) \circ \hat{E}^{a_2t}(\cdot,\gamma_2)  \circ \dots \circ \hat{E}^{a_n t}(X,\gamma_n)$$
   (note that the order of the composition is irrelevant since the different terms are commuting).
  \item For $(X, \lambda) \in \Tgn \times \ML$, the flow is obtained by continuous extension: since the space of weighted multicurves is dense in the space of measured laminations, the earthquake along $\lambda \in \ML$ is the limit of earthquake along weighted multicurves converging to $\lambda$.
\end{itemize}
Because $\hat{E}^t(\cdot,\cdot)$ commutes with the natural $Mod(S)$-action, we have a flow $E^t([\cdot,\cdot])$ on $\PMgn$ also called earthquake flow. For more details on $E^t([\cdot,\cdot])$, see \cite{NielsenRealizationPro}.

\subsection{Lipschitz functions on $\PMgn$}

\label{Norm_Lip}

We need to define distances in order to talk about Lipschitz functions. For this sake, we consider metrics on $\Tgn$ and on $\ML$.

\begin{dfnt}
  The asymmetric Thurston distance (cf. \cite{thurston1998minimal}) on $\Tgn$ between $Y$ and $Y'$ is \[
  d_{Th}^{asym}(Y,Y'):= \log \left( \underset{\lambda \in \ML}{\sup}\frac{l_{\lambda}(Y)}{l_{\lambda}(Y')} \right)
  \]
  By symmetrizing $d_{Th}^{asym}$, we get the Thurston distance:
  \[
d_{Th}(Y,Y'):=\max(d_{Th}^{asym}(Y,Y'),d_{Th}^{asym}(Y',Y))
  \]
\end{dfnt}

Next, we fix a pants decomposition $\gamma_1,...,\gamma_{3g-3}$ in order to introduced a distance on the space of measured laminations.
If we denote by $i(\cdot,\cdot)$ the intersection number and by $t_{\cdot}(\cdot)$ the twisting number on $\ML \times \ML$, then we have the following theorem \cite{penner1992combinatorics}:

\begin{thm}[Dehn-Thurston coordinates]
  \label{DTcoord}
 The map
 \begin{align*}
   \ML & \to  (\mathbb{R}_{>0} \times \mathbb{R})^{3g-3} \cup (\{0\} \times \mathbb{R}_{\geq 0})^{3g-3} \\
   \lambda & \mapsto  (i(\gamma_1,\lambda),t_{\gamma_1}(\lambda),...,i(\gamma_{3g-3},\lambda),t_{\gamma_{3g-3}}(\lambda))
 \end{align*}
 is a bijection.
\end{thm}

Using this system of coordinates, we can pull back the $L^{\infty}$ norm from $\mathbb{R}^{6g-6}$ to $\ML$ to get a distance $d_{lam}$.

In this context, we obtain a distance on $\Tgn \times \ML$ by setting $d_{Th} \times d_{lam}((X,\lambda),(X',\lambda')) =\max(d_{Th}(X,X'),d_{lam}(\lambda,\lambda'))$.

Finally, we obtain a distance on $\PMgn$
 with the formula
 \[
d_{\PMgn}([Y,\lambda],[Y',\lambda'])= \underset{h \in Mod(S)}{\inf} d_{Th} \times d_{lam}((Y,\lambda),(h.Y',h.\lambda')).
\]

In the sequel, we will consider the space of bounded Lipschitz functions on $\PMgn$ equipped with the following norm :\[
\|f\|_{Lip} = \|f\|_{L^{\infty}} + \underset{ {
\substack{x,y \in \PMgn, \\
        x \ne y}}}{\sup} \frac{|f(x)-f(y)|}{d_{\PMgn}(x,y)}.
\]

\subsection{The systole function}

\begin{dfnt}
On $\Tgn$, the quantity $\underset{\delta \in \Gamma}{\inf} l_{\delta}(X) := l_{sys}(X)$ is always positive. It is called the systolic length of $X$ and any curve $\gamma$ realising this minimum is called a systole of $X$. (In general, $X$ could have multiple systoles.)

Observe that $l_{sys}(X)=l_{sys}(h.X)$ for every $h \in Mod(S)$. In particular, $l_{sys}$ is also well-defined on $\Mgn$ (and we will use the same notation for both functions on $\Tgn$ and $\Mgn$).
\end{dfnt}

\begin{lem}
  The function $l_{sys}$ is bounded on $\Tgn$ by a constant $K^{sys}_{g}$.
\end{lem}

\begin{proof}
  Any $X \in \Tgn$ has hyperbolic area equal to $-2 \pi \chi(S)$.
   So, if we pick any point $x \in X$, the set $\{y \in X, d(y,x)\leq R\}$ can't be homeomorphic to a disk for a large enough $R$ (otherwise its area will eventually become bigger than the area of the surface).
    This give an upper bound for the systole function which depends only on the topology of the surface.
\end{proof}

We need also to control the regularity of this function.

\begin{lem}
\label{SysLip}
For each $\epsilon >0$, the function $l_{sys}$ is Lipschitz on $\Mgn^{\epsilon}=\{ l_{sys}(X) \geq \epsilon, X \in \Mgn \}$ (and we denote this Lipshitz constant $C^{Lip}_{sys,\epsilon}$).
\end{lem}

\begin{proof}
  Take $[Y],[Y'] \in \Mgn^{\epsilon}$, a systole $\gamma \in \Gamma$ of $Y$, a systole $\gamma' \in \Gamma$ of $Y'$. We have
  $$ e^{-d_{Th}(Y,Y')}l_{\delta}(Y)\leq l_{\delta}(Y') \leq e^{d_{Th}(Y,Y')} l_{\delta}(Y), \forall \delta \in \Gamma.
  $$

  Hence, $$
  \left\vert 1 - \frac{l_{\gamma}(Y)}{l_{\gamma}(Y')} \right\vert \leq \max( 1-e^{-d_{Th}(Y,Y')},e^{d_{Th}(Y,Y')} -1 ).
  $$

  On the other hand,
\begin{equation*}
    \begin{array}{ll}
      |l_{sys}(Y)-l_{sys}(Y')|   & =  l_{sys}(Y')|1-\frac{l_{sys}(Y)}{l_{sys}(Y')}| \\
      & \leq  K^{sys}_{g} \left\vert 1-\frac{l_{\gamma}(Y)}{l_{\gamma'}(Y')} \right\vert \\
      & =  K^{sys}_{g} \left\vert 1-\frac{l_{\gamma}(Y)}{l_{\gamma}(Y')} +\frac{l_{\gamma}(Y)}{l_{\gamma}(Y')} -\frac{l_{\gamma}(Y)}{l_{\gamma'}(Y')} \right\vert \\
      & \leq  K^{sys}_{g} \left\vert 1-\frac{l_{\gamma}(Y)}{l_{\gamma}(Y')} \right\vert + K^{sys}_{g} \left\vert \frac{l_{\gamma}(Y)}{l_{\gamma}(Y')} -\frac{l_{\gamma}(Y)}{l_{\gamma'}(Y')} \right\vert \\
       & \leq  K^{sys}_{g} \left\vert 1-\frac{l_{\gamma}(Y)}{l_{\gamma}(Y')} \right\vert + \frac{K^{sys}_{g} l_{\gamma}(Y)}{l_{\gamma}(Y')} \left\vert 1 -\frac{l_{\gamma}(Y')}{l_{\gamma'}(Y')} \right\vert \\
       & \leq  K^{sys}_{g} \left\vert 1-\frac{l_{\gamma}(Y)}{l_{\gamma}(Y')} \right\vert + \frac{(K^{sys}_{g})^2 }{\epsilon} \left\vert 1 -\frac{l_{\gamma}(Y')}{l_{\gamma'}(Y')} \right\vert.

    \end{array}
\end{equation*}

  Since $
  l_{\gamma'}(Y') \leq l_{\gamma}(Y')
  $ (as $\gamma'$ is a systole of $Y'$) and $$
l_{\gamma}(Y') \leq e^{d_{Th}(Y,Y')} l_{\gamma}(Y) \leq e^{d_{Th}(Y,Y')} l_{\gamma'}(Y) \leq e^{2 d_{Th}(Y,Y')} l_{\gamma'}(Y'),
  $$
  we also have $$
  \left\vert 1 - \frac{l_{\gamma}(Y')}{l_{\gamma'}(Y')} \right\vert \leq e^{2 d_{Th}(Y,Y')}-1.
  $$
    By combining these estimates with the bound $$
    |l_{sys}(Y)-l_{sys}(Y')| \leq 2 K^{sys}_{g},
    $$
    we get
    $$
\frac{|l_{sys}(Y)-l_{sys}(Y')|}{d_{Th}(Y,Y')} \leq \min( \frac{2 K^{sys}_{g}}{d_{Th}(Y,Y')} , K^{sys}_{g} \frac{\max( 1-e^{- d_{Th}(Y,Y')},e^{d_{Th}(Y,Y')} -1 )}{d_{Th}(Y,Y')} + \frac{(K^{sys}_{g})^2}{\epsilon} \frac{e^{2 d_{Th}(Y,Y')}-1}{d_{Th}(Y,Y')}).
    $$
    This completes the proof because the maximum over all $d_{Th}(Y,Y') \in ]0, +\infty[$ of the right hand side is finite and it gives an upper bound of the Lipschitz constant of $l_{sys}(\cdot)_{\vert \Mgn^{\epsilon}}$.
\end{proof}

Next, we recall the definition of the function $B(\cdot)$ indicating how the volume of the unit ball of the space of measured lamination changes in the moduli space.

\begin{dfnt}
  \label{DefB}
  For $X \in \Tgn$ we denote by $B(X) := \mu_{Th}'(\lambda \in \ML,l_{\lambda}(X) \leq 1)$. This function is $Mod(S)$-invariant and therefore well-defined on $\Mgn$. We keep the same notation $B(\cdot)$ for the function induced on $\Mgn$.
\end{dfnt}

\begin{lem}
  The function $B$ has a strictly positive lower bound on $\Mgn$.
\end{lem}

\begin{proof}
  First, the function $B$ is continuous  and never vanishes according to Proposition 3.2 of \cite{mirzakhani2008growth}. Secondly, on the complement of $\Mgn^{\epsilon}$, according to Proposition 3.6 of \cite{mirzakhani2008growth} we have \[
   1 \leq  \underset{
   {\substack{\gamma \in \Gamma\\
           l_{\gamma}(X) \leq \epsilon}}
   }{\prod} \frac{1}{l_{\gamma}(X) \log(l_{\gamma}(X))} \leq B(X)
   \]
   for $\epsilon$ small enough. This completes the proof.
\end{proof}

Finally, we need to control the difference of two intersection numbers with a given curve. In this direction, we use a lemma of M. Rees.

\begin{lem}[cf \cite{reesalternative}, Lemma 1.10]
\label{LipIntSim}
  For a simple closed curve $\gamma$ and two measured laminations $\lambda$ and $\lambda'$ we have
  $$
  |i(\lambda,\gamma) - i(\lambda',\gamma)| \leq C^{Lip}_{int,\gamma} d_{lam}(\lambda,\lambda')
  $$
\end{lem}

\section{Rate of mixing of the earthquake flow}
\label{RoM}

Given $\rho$ a positive real number, $\lambda \in \mathcal{ML}$, $X \in \Tgn$  and $\gamma \in \Gamma$, let
$$
J_{\gamma}^{\lambda,X}(\rho) := \{t \in \mathbb{R} , l_{\gamma}(\hat{E}_t(X,\lambda)) \leq \rho \}$$

and

 $$\epsilon_{\gamma}^{\lambda,X} = \underset{t \in \mathbb{R}}{\min} \  l_{\gamma}(\hat{E}_t(X,\lambda)).$$

We will use two results due to S. Kerckhoff \cite{NielsenRealizationPro} and Y. Minsky and B. Weiss \cite{articleMetW}. The first one bounds the variation of the length of a multicurve during the earthquake flow in terms of the intersection number between the multicurve and the lamination directing the earthquake.

\begin{lem}[cf. \cite{NielsenRealizationPro} Corollary 3.4] Let $(X, \lambda) \in \PTgn$ and $\gamma \in \Gamma$, the variation of the length of the curve is bounded by:
$$ \left\vert \frac{d}{dt}l_{\gamma}(\hat{E}^t(X,\lambda)) \right\vert <  i(\lambda,\gamma) .
$$
\end{lem}

The second one frames the length of a multicurve during the earthquake between two parallel lines.

\begin{lem}[cf. \cite{articleMetW}, Lemma 5.2]
  \label{LemFra}

There are constants $\rho \leq K^{sys}_g$ and $C_{lem \ref{LemFra}}$, depending only on $S$, such that for any $(X, \lambda) \in \PTgn$,$\gamma \in \Gamma$
and all $t\in J_{\gamma}^{\lambda}(\rho)$, \[
i(\lambda,\gamma) |t-t_\gamma |-C_{lem \ref{LemFra}} \epsilon_{\gamma}^{\lambda,X} \leq l_{\gamma}(\hat{E}^t(X,\lambda)) \leq i(\lambda,\gamma) |t- t_\gamma | + \epsilon_{\gamma}^{\lambda,X}
\]

\end{lem}

\begin{figure}[h]
  \hspace*{2cm}
\includegraphics[scale=0.4]{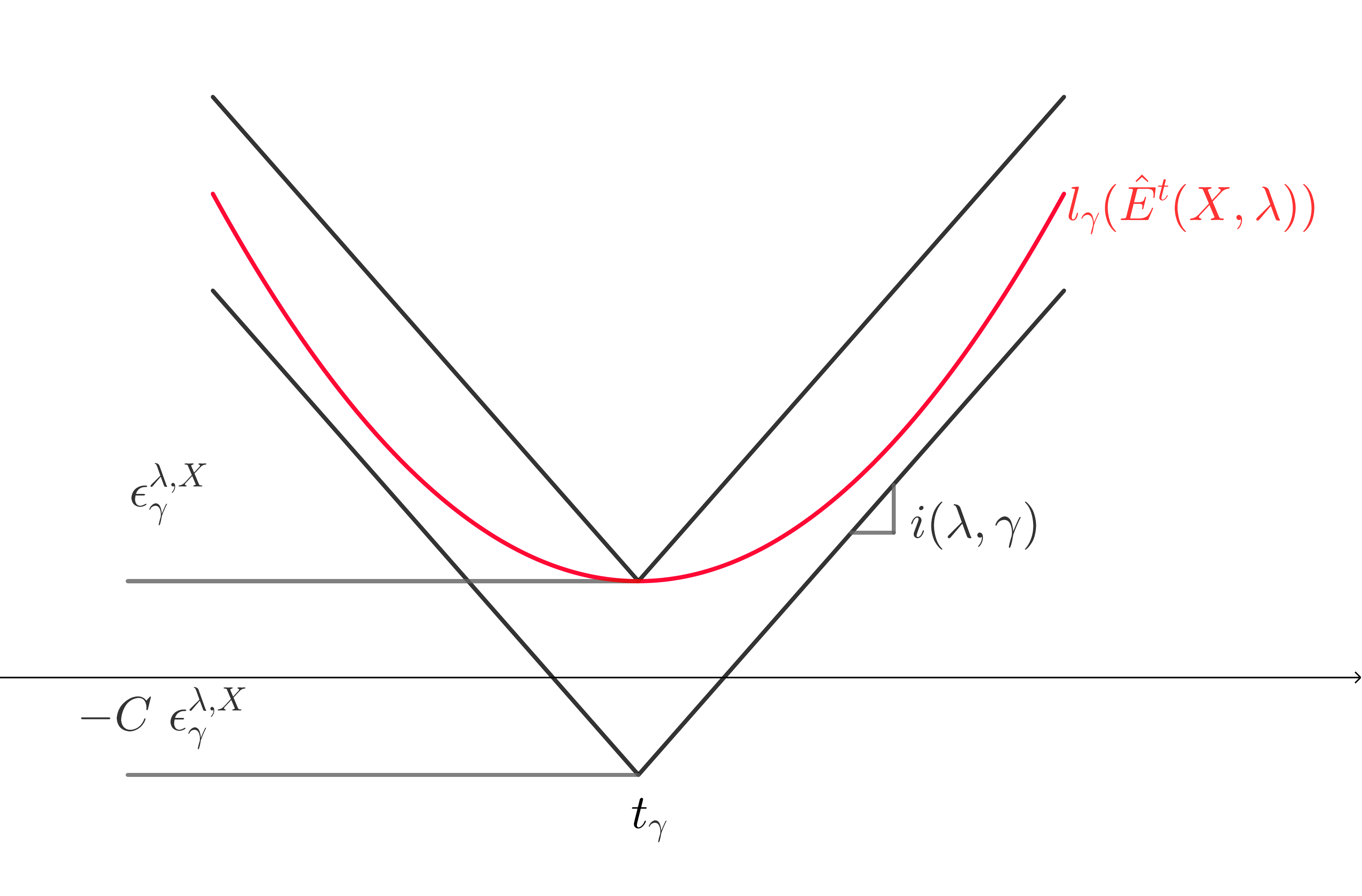}
\caption{The length of a curve is controled along the earthquake flow during a small time near its minimum.}
\centering
\end{figure}

We now fix $\epsilon_0$ such that $ \epsilon_0(2+C_{lem \ref{LemFra}}) <\frac{\rho}{2}  $, and we consider a parameter $\mu > 0$.
Since we will keep $\epsilon_0$ fixed and will decrease $\mu$ in what follows, we will often omit the dependence on $\epsilon_0$ of several constants in most estimates below.

 Let $$ \begin{array}{l}
 \Omega_{\epsilon_0,\mu}=\{ [X,\lambda] \in \PMgn ,l_{sys} (X) \in  [\frac{\epsilon_0}{2};\epsilon_0]
 \text{ and } i(\lambda,\gamma) \in [0,\mu] \text{ when } l_{\gamma}(X) \in [\frac{\epsilon_0}{2};\epsilon_0] \}
\end{array}
$$

 Moreover, let $D$ be the set $\{ [X,\lambda] \in \mathcal{PM}^1,l_{sys}(X) > \rho / 2 \} $. Obviously this set is non-empty (because $\rho < K_{sys}^g$) and has a positive $\nu$-measure.

Given $[X,\lambda] \in \Omega_{\epsilon_0,\mu}$ and $\gamma \in \Gamma$ such that $\tilde{\epsilon} := l_{\gamma}(X) \in [\frac{\epsilon_0}{2};\epsilon_0]$, if $\epsilon_{\gamma}^{\lambda,X}=\epsilon_\gamma$ is attained along the earthquake flow orbit $(X,\lambda)$ at time $t_{\gamma}$,
then the inequalities in Lemma \ref{LemFra} for $t=0$ yield
$$
\tilde{\epsilon}-\epsilon_\gamma \leq i(\lambda,\gamma) |t_\gamma | \leq \tilde{\epsilon} + C_{lem \ref{LemFra}} \epsilon_\gamma.
$$
Because $0 < \epsilon_\gamma \leq \tilde{\epsilon}$ and $[X,\lambda] \in \Omega_{\epsilon_0,\mu}$, we have $i(\lambda,\gamma) |t_\gamma| \in [0; \tilde{\epsilon}(1+C_{lem \ref{LemFra}})]$. By combining these estimates with Lemma \ref{LemFra}, it follows that, for any $t \in J_{\gamma}^{\lambda,X}(\rho)$,
one has
$$
l_\gamma (\hat{E}_t (X,\lambda)) \leq \tilde{\epsilon} (2+C_{lem \ref{LemFra}}) +i(\lambda,\gamma) |t|
 \leq \epsilon_0 (2+C_{lem \ref{LemFra}}) + \mu |t|
 \leq \frac{\rho}{2} + \mu |t|.
$$

Hence, for all $|t| \leq \frac{\rho}{2 \mu} := t_{lim}$, we have that $l_\gamma(\hat{E}_t(X,\lambda)) \leq \rho$. In particular, $E_t([X,\lambda]) \notin D$ for all $|t| \leq t_{lim}$.

Therefore, if $f_{\epsilon_0,\mu}$ is a Lipschitz positive function with support included in $\Omega_{\epsilon_0,\mu}$
and $\int f_{\epsilon_0,\mu} > K_0 \nu(\Omega_{\epsilon_0,\mu})$, and $g$ is a positive Lipschitz function with support in $D$, $K_1= \int g$ and $K_2 = \| g \|_{Lip}$, then
$$
\left\vert \int f_{\epsilon_0,\mu} \circ E_{t_{lim}} g d \nu - \int f_{\epsilon_0,\mu} d \nu \int g d \nu\right\vert = K_1 \int f_{\epsilon_0,\mu} > K_0 K_1 \nu(\Omega_{\epsilon_0,\mu}).
$$

\newpage

Thus, if the earthquake flow were polynomially mixing with degree $d$ in the sense of the inequality \ref{eq2}, then:

\begin{equation}
  \label{equGe}
  \nu(\Omega_{\epsilon_0,\mu})K_1 K_0 \leq K_1 \int f_{\epsilon_0,\mu} d \nu
  < K_2 \| f_{\epsilon_0,\mu} \|_{Lip} \frac{C}{t_{lim}^d}.
\end{equation}

In the remainder of this section, we will contradict this estimate (for $d> 6g-5$) by constructing $f_{\epsilon_0,\mu}$ with a small Lipschitz norm relatively to the volume ($\nu$-mass) of $\Omega_{\epsilon_0,\mu}$.

\subsection{Estimate of  \texorpdfstring{ $\| f_{\epsilon_0,\mu} \|_{Lip}$ }{the Lipschitz norm} }

We will define $\hat{f}_{\epsilon_0,\mu}$ on $\PTgn$ as a product of piecewise linear function as follows: $
\hat{f}_{\epsilon_0,\mu}=g_{\epsilon_0}j_{\mu}
$
with

$$
g_{\epsilon_0}(X) = \left \{
\begin{array}{lll}
0 & \text{if} & l_{sys}(X) \notin [\epsilon_0/2;\epsilon_0]\\
1 & \text{if} & l_{sys}(X) \in [\frac{4 \epsilon_0}{6};\frac{5 \epsilon_0}{6}]\\
\frac{6}{\epsilon_0}(l_{sys}(X)-\frac{\epsilon_0}{2}) & \text{if} & l_{sys} \in [\frac{\epsilon_0}{2};\frac{4 \epsilon_0}{6}]\\
\frac{-6}{\epsilon_0}(l_{sys}(X)-\epsilon_0) & \text{if} & l_{sys} \in [\frac{5 \epsilon_0}{6};\epsilon_0]
\end{array}
\right .
$$
and
$$
j_{\mu}(X,\lambda) = \prod_{\gamma \in \Gamma, \frac{\epsilon_0}{2} < l_{\gamma}(X) \leq \epsilon_0} j_{\mu,\gamma}(\lambda)
$$
where
$$
j_{\mu,\gamma}(\lambda) = \left \{
\begin{array}{lll}
0 & \text{if} & i(\lambda,\gamma) > \mu\\
1 & \text{if} & i(\lambda,\gamma) \leq \mu/2\\
\frac{2 }{ \mu}(\mu-i(\lambda,\gamma)) & \text{if} & i(\lambda,\gamma) \in [\mu /2 ; \mu].
\end{array}
\right .
$$

We will show that :

\begin{lem}
  $\hfem$ is Lipschitz and, for $\mu$ small enough,
  $$
  \begin{array}{lll}
  \| \hfem \|_{Lip} & \leq & \frac{C^{Lip}_{f}}{ \mu}\\
  \end{array}.
  $$
\end{lem}
\begin{proof}

As $\| \hat{f}_{\epsilon_0,\mu}\|_{L^{\infty}} \leq 1$, our task is reduced to bound
$$
\underset{\substack{(X,\lambda) \in \PTgn, \\
(X_n,\lambda_n) \in \PTgn \setminus \{(X,\lambda)\}, \\
(X_n,\lambda_n) \to (X,\lambda)
}}{\sup} \,
\lim_{n \to \infty} \frac{|\hat{f}_{\epsilon_0,\mu}(X,\lambda)-\hat{f}_{\epsilon_0,\mu}(X_n,\lambda_n)|}{d_{Th}\times d_{lam}((X,\lambda),(X_n,\lambda_n))}.
$$
 In order to avoid difficulties which can occur when the curves in the product defining $j_{\mu}$ are not the same for $X$ and $X_n$, we divide our analysis into three cases.

If $(X,\lambda) \not\in supp(\hfem)$ then eventually $(X_n,\lambda_n) \not\in supp(\hfem)$ and we are done.

If $(X,\lambda) \in supp(\hfem)$ let $d_X = min(l_{\gamma}(X)- \frac{\epsilon_0}{2}, \epsilon_0 - l_{\gamma}(X))$ where $\gamma$ a curve is a systole.

If $d_X > 0$ eventually $l_{\gamma}(X_n) \in [\frac{\epsilon_0}{2}, \epsilon_0]$ for all $n$ large enough.
In this case we have
$$
|\hfem(X,\lambda)-\hfem(X_n,\lambda_n)| \leq |g_{\epsilon_0}(X) -g_{\epsilon_0}(X_n) | \|j_{\mu}\|_{L^{\infty}} + | j_{\mu}(X,\lambda)-j_{\mu}(X_n,\lambda_n)|\|g_{\epsilon_0}\|_{L^{\infty}}.
$$
Since $\|g_{\epsilon_0}\|_{L^{\infty}}=\|j_{\mu}\|_{L^{\infty}}=1$, Lemma \ref{SysLip} ensures $\|g_{\epsilon_0}\|_{Lip} = \frac{6}{\epsilon_0} C^{Lip}_{sys,\epsilon_0}+1$.
On the other hand, Lemma \ref{LipIntSim} guarantees that
$$
|j_{\mu}(X,\lambda)-j_{\mu}(X_n,\lambda_n)| \leq \frac{2}{\mu} |i(\lambda,\gamma)-i(\lambda_n,\gamma)| \leq \frac{2 C^{Lip}_{int,\gamma} }{\mu}d_{lam}(\lambda,\lambda_n).
$$
Because $X \in \Tgn$ with $l_{sys}(X) \in [\epsilon_0 /2 , \epsilon_0]$ has a finite number of curve of lenght less than $\epsilon_0$, we can define
$$
C^{Lip}_{int,X} = \underset{l_{\gamma}(X) \leq \epsilon_0}{\max}C^{Lip}_{int,\gamma}.
$$
Therefore, if $\Delta$ is the intersection of a fundamental domain of the $Mod(S)$-action on $\Tgn$ intersect with $l_{sys}^{-1}([\epsilon_0 /2 , \epsilon_0])$, then we can use the compactness of $\Delta$ to define
$$
C^{Lip}_{int} = \underset{X \in \Delta}{\max} \ C^{Lip}_{int,X}
$$
 and to conclude this second case.

Finally, in the third case when $d_X=0$, one has $\hfem(X,\lambda)=g_{\epsilon_0}(X)=0$.
 Since $\hfem(X_n,\lambda_n) \leq g_{\epsilon_0}(X_n) \leq \|g_{\epsilon_0} \|_{Lip} d_{Th}(X,X_n)$, we are also done in this case.

By taking these three cases into account, we see that there is a constant $C^{Lip}_{f}$ such that
$$
\begin{array}{lll}
\| \hfem \|_{Lip} & \leq & \frac{C^{Lip}_{f}}{ \mu}\\
\end{array}
$$
for $\mu$ small enough.
\end{proof}
The function $\hfem$ is $Mod(S)$-invariant and therefore it naturally descends to a function $\fem$ on $\PMgn$ with the same Lipschitz norm.

\subsection{Estimate of  the volume of \texorpdfstring{$\Omega_{\epsilon_0,\mu}$}{the considered places} }

We will now estimate the volume of $\Omega_{\epsilon_0,\mu}$.

If $\gamma_1$ and $\gamma_2$ are simple closed curves the collar lemma states (see Corollary 3.4 of \cite{martelli2016introduction}) that:
 $$
 i(\gamma_1,\gamma_2) \leq  \frac{l_{\gamma_1}(X)}{R(l_{\gamma_2}(X))}
 $$
  where $R$ is a decreasing  function with $\underset{x \to 0}{\lim} \ R(x) = + \infty$ and $\underset{x \to + \infty}{\lim} R(x) = 0$. Then, we can partially extend this inequality for multicurves (by linearity) and for measured lamination (by continuity), that is:
 $$
 i(\gamma,\lambda) \leq  \frac{l_{\lambda}(X)}{R(l_{\gamma}(X))}.
 $$

  Now in our case, if $\gamma$ is such that $\frac{\epsilon_0}{2} < l_{\gamma}(X) \leq \epsilon_0 $
  we have for every curve $\delta$:
  $$
  i(\gamma,\delta) \leq \frac{l_{\delta}(X)}{R(\epsilon_0)}.
  $$

Assuming that $\mu$ is small enough so that $\mu R(\epsilon_0) \leq 1$, we get
$$\{\delta, l_{\delta}(X) \leq  \mu R(\epsilon_0 ) L \} \subset \{ \delta, i(\delta,\gamma) \leq \mu L ,l_{\delta}(X) \leq L \}$$
for any $L > 0$.
Indeed $i(\delta,\gamma)\leq \frac{1}{R(\epsilon_0 )} l_{\delta}(X) \leq \mu L$.
So, if $X \in \Tgn$ and $\delta_1, \dots ,\delta_k$ are the curves such that $\frac{\epsilon_0}{2} < l_{\delta_j}(X) \leq \epsilon_0 $, we get that:
\begin{equation*}
  \begin{array}{lll}
    (\mu R(\epsilon_0 ))^{6g-6} \frac{ \# \{\delta, l_{\delta}(X) \leq  \mu R(\epsilon_0) L \} }{(\mu R(\epsilon_0) L)^{6g-6}} & = & \frac{ \# \{\delta, l_{\delta}(X) \leq  \mu R(\epsilon_0) L \} }{L^{6g-6}} \\
     & \leq & \frac{\# \{ \delta, i(\delta,\delta_j) \leq \mu L, j \in [1,k] ,l_{\delta}(X) \leq L \}}{L^{6g-6}}
  \end{array}
\end{equation*}
\newpage
and, by taking the limit $L \to \infty$, we find that :
$$
(\mu R(\epsilon_0 ))^{6g-6} B(X) \leq \mu_{Th}(X){\{ \lambda \in \mathcal{ML}, i(\lambda,\delta_j) \leq \mu , j \in [1,k] \}},
$$
where $B(X)$ is the Thurston volume of the unit ball in the space of lamination (cf. Definition \ref{DefB} above).

In this way, we derive that
\begin{equation*}
  \begin{array}{ll}
    \int \fem d \nu & \geq  \int \mathds{1}_{l_{sys}(X) \in [\frac{4 \epsilon_0}{6},\frac{5 \epsilon_0}{6}]} \mathds{1}_{j_{\mu}(X,\lambda) = 1} d \nu \\
     & = \int_{\Theta} \mathds{1}_{l_{sys}(X) \in [\frac{4 \epsilon_0}{6},\frac{5 \epsilon_0}{6}]} \int_{\ML} \mathds{1}_{i(\lambda,\delta_j) \leq \mu/2, j \in [1,k]} d \mu_{Th}(X)(\lambda) d \mu_{WP}(X) \\
     & \geq  K_{vol} \mu^{6g-6}
  \end{array}
\end{equation*}
where $\Theta$ is a fondamental domain of the $Mod(S)$-action on $\Tgn$ and
$$K_{vol}=\left( \frac{R(\epsilon_0)}{2} \right)^{6g-6} \int_{\Theta} \mathds{1}_{l_{sys}(X) \in [\frac{4 \epsilon_0}{6},\frac{5 \epsilon_0}{6}]} B(X) d \mu_{WP}(X).$$

\subsection{End of the proof of the main result}

We now prove Theorem \ref{Thprincipal}. Keeping the same notations and objects as before, we suppose that the earthquake flow is polynomially mixing in the sense of the inequality \ref{eq2}.

Recall that the inequality \eqref{equGe} yields
\begin{equation*}
\nu(\Omega_{\epsilon_0,\mu}) < K \| f_{\epsilon_0,\mu} \|_{Lip} \mu^d,
\label{poly}
\end{equation*}
where $d \in \mathbb{R^{+*}}$ and $K$ a synthetic constant.

By injecting the estimates on $\|\hfem\|_{Lip}$ and $\nu(\Omega_{\epsilon_0,\mu})$ from the two previous subsections, we have:

$$
\mu^{6g-5} < K \mu^d.
$$

Since this inequality is false (for $\mu$ small enough) when $d > 6g-5$ the proof of Theorem \ref{Thprincipal} is complete.

\section*{Acknowledgements}
The author warmly thank Carlos Matheus for his guidance on the topic, and his knowledgeable feedbacks. We are also very grateful to Mingkun Liu for early discussions and insightful comments.
\bibliographystyle{plainnat}
\bibliography{biliographie.bib}

\end{document}